\documentclass[10pt]{article}
\font\smallit=cmti10
\font\smalltt=cmtt10

\usepackage{amssymb,latexsym,amsmath,epsfig,amsthm} 

\makeatletter

\renewcommand\section{\@startsection {section}{1}{\z@}
{-30pt \@plus -1ex \@minus -.2ex}
{2.3ex \@plus.2ex}
{\normalfont\normalsize\bfseries\boldmath}}

\renewcommand\subsection{\@startsection{subsection}{2}{\z@}
{-3.25ex\@plus -1ex \@minus -.2ex}
{1.5ex \@plus .2ex}
{\normalfont\normalsize\bfseries\boldmath}}

\renewcommand{\@seccntformat}[1]{\csname the#1\endcsname. }

\makeatother

\newtheorem{theorem}{Theorem}
\newtheorem{lemma}{Lemma}

\newtheorem{corollary}{Corollary}

\theoremstyle{definition}

\begin{document}

%
\begin{center}
\uppercase{\bf Extrema of Luroth Digits and a zeta function limit relations}
\vskip 20pt
{\bf Jayadev S. Athreya\footnote{Supported by National Science Foundation Grant DMS 2003528}}\\
{\smallit Department of Mathematics and Department of Comparative History of Ideas, University of Washington, Seattle, WA, USA.}\\
{\tt jathreya@uw.edu}\\ 
{\bf Krishna B. Athreya}\\
{\smallit Department of Mathematics and Department of Statistics, Iowa State University, Ames, IA, USA.}\\
{\tt kbathreya@gmail.com}\\ 
\end{center}
\vskip 20pt
\centerline{\smallit Received: , Revised: , Accepted: , Published: } 
\vskip 30pt

%
%
%
%
%
%
%
%

\centerline{\bf Abstract}
\noindent
We describe how certain properties of the extrema of the digits of Luroth expansions lead to a probabilistic proof of a limiting relation involving the Riemann zeta function and the Bernoulli triangles. We also discuss trimmed sums of Luroth digits. Our goal is to show how direct computations in this case lead to formulas and some interesting discussions of special functions. 

\pagestyle{myheadings}
\markright{\smalltt INTEGERS: 21 (2021)\hfill}
\thispagestyle{empty}
\baselineskip=12.875pt
\vskip 30pt

\section{Introduction} 

Let $X_1, X_2, \ldots X_k, \ldots$ be a sequence of independent, identically distributed (IID) random variables taking values in the positive integers $\mathbb N =\{1, 2, 3, \ldots \}$ with $$P(X_1 = n) = \frac{1}{n(n+1)}, n \geq 1.$$ We call this distribution the \emph{Luroth} distribution. The \emph{Luroth map} $L: (0, 1] \rightarrow (0, 1]$ is the piecewise-linear map given by $$L(x) = N(x) ( (N(x)+1) x - 1),$$ where $$N(x) =\left \lfloor \frac 1 x \right \rfloor.$$ It is a nice exercise to check that $L$ preserves Lebesgue measure $m$, and that if $U$ is a uniform $(0, 1]$ random variable, the sequence $X_k = N(L^{k-1} (U))$ is a sequence of IID random variables with the above distribution, since $$P(X_1 = n) = m\left(N^{-1}(n)\right) = m\left(\frac{1}{n+1}, \frac{1}{n}\right] = \frac{1}{n(n+1)}.$$ 

This map and its associated digit expansion was introduced by Luroth in~\cite{Luroth}.  Subsequently, Luroth random variables have been extensively studied,  since they provide an attractive motivating example at the intersection of many questions on number theory, probability theory, and dynamical systems and ergodic theory. The Luroth map is in some sense a kind of linearized version of the Gauss map (see Section~\ref{sec:Gauss}), where, as we remark above, the resulting digit expansion of numbers have the appealing property that the digits are \emph{independent} random variables. 

This makes it often possible to do very explicit computations with Luroth random variables. See~\cite{JDv} for an introduction to the ergodic properties of the map $L$, ~\cite{Galambos} for a discussion of different types of digit expansions, and, for example~\cite{Giuliano} and~\cite{Tan} and the references within for more recent work on Luroth and related expansions from the probabilistic and dynamical perspectives respectively.

Our results center on understanding the limiting behavior of \emph{maxima} of sequences of IID Luroth random variables, and the convergence behavior of appropriately centered and scaled sums. These are classical topics in the theory of sequences of random variables, and focusing on this setting leads to some interesting computations.

Using the fact that the Luroth random variables are \emph{heavy-tailed}, we will show that the probability $\rho_k$ that the maximum $M_k$ of the first $k$ elements $\{X_1, \ldots, X_k\}$ of a a sequence of Luroth variables is \emph{unique} tends to $1$ as $k \rightarrow \infty.$ Our main new result, Theorem~\ref{theorem:rhok}, shows how to \emph{explicitly} compute this probability $\rho_k$, giving an interesting relationship between this probability, the Riemann zeta function, and partial sums of binomial coefficients which are entries in what is known as Bernoulli's triangle. Using the fact that $\rho_k \rightarrow 1$, we obtain a new limiting relation, Corollary~\ref{cor:zeta}, between these coefficients and the Riemann zeta function.

Our other main result considers the behavior of the sample sums $$S_k = \sum_{i=1}^k X_i.$$ Since $E(X_1) = \infty$, the law of large numbers shows that $S_k/k$ tends to infinity with probability one. Inspired by work of Diamond-Vaaler, who studied the corresponding \emph{trimmed sum} for continued fraction expansions, and using a result of Mori~\cite{Mori} we show in Theorem~\ref{theorem:trimmed} that by removing the maximum, and normalizing by $k\log k$, we have, with probability $1$, as $k \rightarrow \infty$,  $$\frac{S_k-M_k}{k\log k} \rightarrow 1.$$ Note that the asymptotic uniqueness result says that asymptotically, we are removing the \emph{unique} largest summand in $S_k$ by subtracting $M_k$. Let $$M_k = \max(X_1, \ldots, X_k)$$ be the sample maximum, and let $$\rho_k = P(\mbox{ there exists } !  1 \le i \le k \mbox{ such that }X_i = M_k)$$ be the probability that this maximum is achieved exactly once.  Our first main observation is the following explicit formula for $\rho_k$. 
\begin{theorem}\label{theorem:rhok} For $k \geq 1$, \begin{equation}\label{eq:rhok} \rho_k = k\left (2^{k-1} + \sum_{j=2}^k T(k-1, k-j) \zeta(j) (-1)^{j+1}\right),\end{equation}  where $\zeta(j) = \sum_{n=1}^{\infty} n^{-j}, j > 1$ is the Riemann zeta function and for $l \geq j$ positive integers, $$T(l, j) = \sum_{i=0}^j \binom{l}{i}.$$
\end{theorem}

By \cite{AF, ESS}, we have $$\lim_{k \rightarrow \infty} \rho_k = 1.$$ Thus, as a corollary, we obtain the following limiting relation involving the coefficients $T(l, j)$ and values of the Riemann zeta function. 

\begin{corollary}\label{cor:zeta} We have
\begin{equation}\label{eq:zeta}\lim_{k \rightarrow \infty}  k\left (2^{k-1} + \sum_{j=2}^k T(k-1, k-j) \zeta(j) (-1)^{j+1}\right) = 1.\end{equation}

\end{corollary}

The coefficients $T(l, j)$, partial sums of binomial coefficients, are entries in what is known as \emph{Bernoulli's triangle}; see~\cite{OEIS-A008949}. Note that $T(l, 0) = 1$, $T(l, l) = 2^l$, and $T(l, l-1) = 2^{l} -1$. Let $$S_k = \sum_{i=1}^k X_i$$ be the sample sum. Since $E(X_1) = \infty$, we have, by the strong law of large numbers, $$\lim_{k \rightarrow \infty} \frac{S_k}{k} = +\infty.$$ However, if we subtract the sample maximum, we have the following almost sure result.

\begin{theorem}\label{theorem:trimmed} With probability $1$, $$\lim_{k \rightarrow \infty}  \frac{S_k - M_k}{k \log k} = 1.$$

\end{theorem}

\noindent This is a consequence of a result of Mori~\cite{Mori}. The proof features a nice appearance of the Lambert $W$-function. There is a (more difficult) analogous result~\cite{DV} for continued fraction digits, where the limiting value is $\log 2$. There are substantial additional difficulties in that setting since continued fraction digits of a randomly chosen $x \in (0, 1)$ do not form an IID sequence.

\section{Uniqueness of Maxima}\label{sec:maxproof}
 Suppose $Y_1, \ldots, Y_k, \ldots$ is a sequence of IID $\mathbb N$-valued random variables with $$P(Y_1 = n) = p_n.$$ Let $$\tau_n = \sum_{m \geq n} p_m = P(Y_1 \geq n).$$ Let $\rho_{k, m}$ denote the probability that the sample maximum $M_k = \max(Y_1, \ldots, Y_k) = m$ and is achieved exactly once. Then $$\rho_{k,m} = k p_m (1-\tau_m)^{k-1}.$$ Note that $\rho_k$, the probability that the sample maximum $M_k$ is achieved exactly once by $Y_1, \ldots, Y_k$, is then given by $$\rho_k = \sum_{m=1}^{\infty} \rho_{k,m}.$$ Also note that (although it is not a probability) we have $$\sum_{k=1}^{\infty} \rho_{k,m} = p_m \sum_{k=1}^{\infty} k(1-\tau_m)^{k-1} = \frac{p_m}{\tau_m^2}.$$ Specializing to Luroth sequences, we have $p_m = \frac{1}{m(m+1)}$ and $\tau_m = \frac{1}{m}$. Thus $$\rho_{k,m} = \frac{k}{m(m+1)} \left( 1 - \frac 1 m \right)^{k-1} = k \frac{(m-1)^{k-1}}{m^k(m+1)} .$$ Let $$Q_k(m) = \frac{1}{k} \rho_{k, m} = \frac{(m-1)^{k-1}}{m^k(m+1)}.$$ Th next lemma shows how to expand $Q_k$ as a partial fraction.

\begin{lemma}\label{lem:partial} We have \begin{equation}\label{eq:qk}Q_k(m) = 2^{k-1} \left( \frac 1 m - \frac{1}{m+1} \right) + \sum_{j=2}^k T(k-1, k-j) (-1)^{j+1} \frac{1}{m^j}.\end{equation}
\end{lemma}

\begin{proof} Note that $$T(k-1, 0) = 1, T(k-1, k-1) = 2^{k-1}, T(k-1, k-2) = 2^{k-1} - 1$$ and $$Q_{k+1}(m) = \left( 1 - \frac 1 m\right) Q_{k}(m).$$ Our proof proceeds by induction. For $k=1$, we have $$Q_1(m) = \rho_{1, m} = \frac{1}{m(m+1)}.$$ The right hand side of  Equation (\ref{eq:qk}) for $k=1$ is $$2^{1-1} \left (\frac 1 m - \frac{1}{m+1} \right) = \frac{1}{m(m+1)},$$ since the index of the sum starts at $j=2$, making the sum empty. Thus the base case ($k=1$) is verified. We now assume that Equation (\ref{eq:qk}) is true for $k$, and we need to prove it for $k+1$. Let $$c(j, k) = (-1)^{j+1} T(k-1, k-j).$$ Note that

\begin{align}\label{eq:induction} Q_{k+1}(m) &= \left( 1 - \frac 1 m\right) Q_{k}(m)  \nonumber
= \left( 1 - \frac 1 m\right) \left(2^{k-1} \left( \frac 1 m - \frac{1}{m+1} \right) + \sum_{j=2}^k c(j,k) \frac{1}{m^j}\right) \\ \nonumber
&= 2^{k-1}\frac{1}{m(m+1)} - \frac{2^{k-1}}{m}\left( \frac 1 m - \frac{1}{m+1}\right) + \sum_{j=2}^{k}\left( \frac{c(j,k) }{m^j} -\frac{c(j-1, k) }{m^{j+1}}\right)\\ \nonumber 
&= 2^k \frac{1}{m(m+1)} - \frac{(2^{k-1} - c(2, k))}{m^2} + \sum_{j=3}^{k+1} \frac{\left(c(j, k)- c(j-1, k) \right)}{m^j} . \\
\end{align}

\noindent Now note that for $j \geq 2$, $$c(j, k) = c(j, k-1) - c(j-1, k-1)$$ and $$c(k+1, k) = 0, c(2, k) = 1-2^{k-1}, c(k,k) = (-1)^{k+1}, c(1, k) = 2^{k-1}.$$ Plugging these into Equation (\ref{eq:induction}), we have our result.
\end{proof}

\subsection{Proof of Theorem~\ref{theorem:rhok}} To prove Theorem~\ref{theorem:rhok}, we note that $$\sum_{m=1}^{\infty} \frac{1}{m(m+1)} = 1;  \sum_{m=1}^{\infty} \frac{1}{m^j} = \zeta(j), j \geq 2.$$ Using  Lemma~\ref{lem:partial}, and summing over $m$, we obtain Equation (\ref{eq:rhok}). To obtain Corollary~\ref{cor:zeta}, it was shown in~\cite{AF, ESS} that $$\rho_k \rightarrow 1 \mbox{ if and only if } p_m/\tau_m \rightarrow 0.$$ In our setting, $\rho_m/\tau_m = \frac{1}{m+1}$, so the condition is satisfied, and we obtain Corollary~\ref{cor:zeta}. In Figure~\ref{fig:rhok}, we show numerical values of $\rho_k$ for $2 \le k \le 40$ (note that $\rho_1 = 1$).

\begin{figure}[h!] \includegraphics[width=0.9\textwidth]{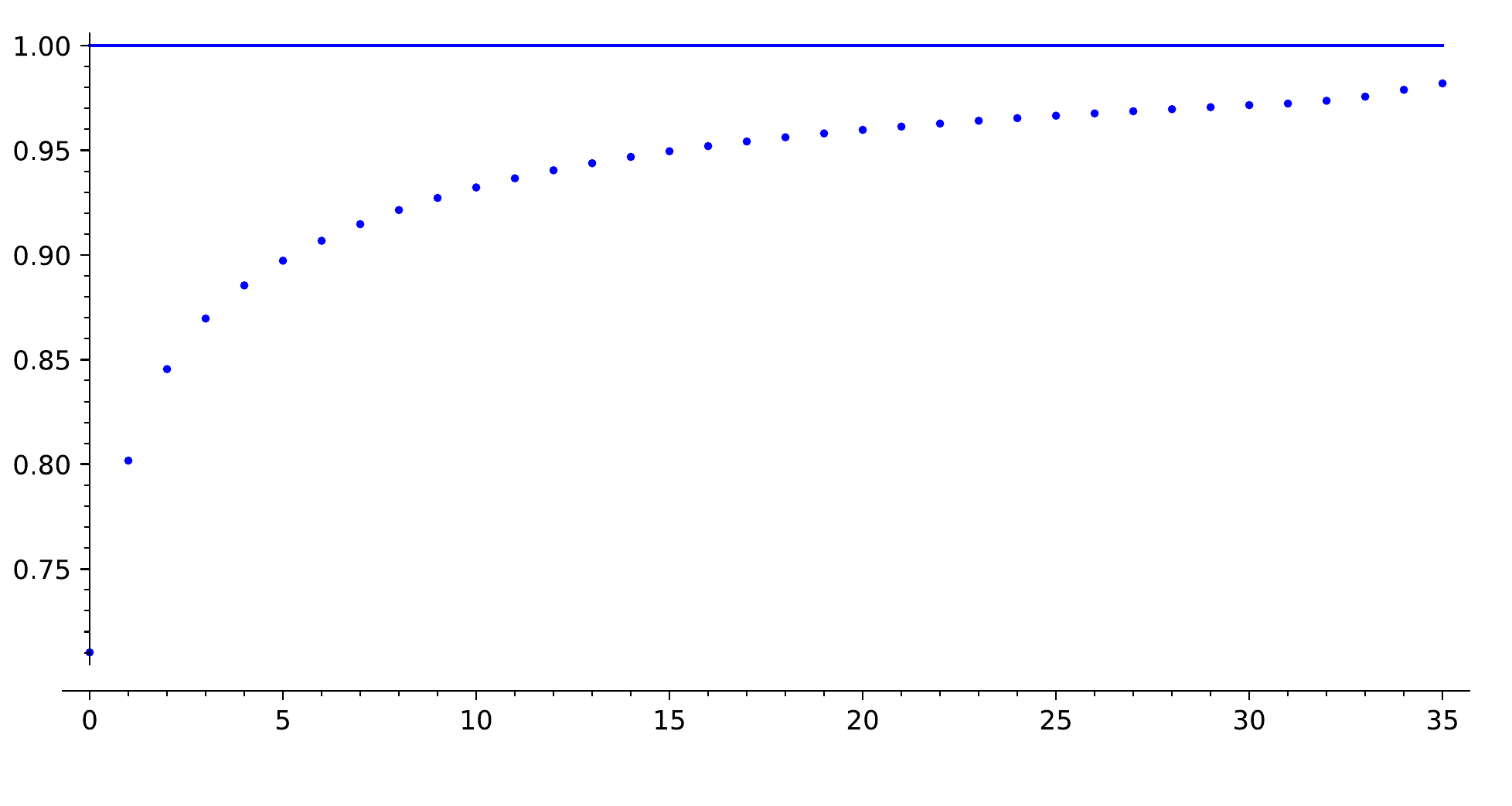}
\caption{$\rho_k$ for $k =2$ to $k=40$. \medskip}\label{fig:rhok}
\end{figure}

\section{Trimmed sums}\label{sec:trimproof}
In this section, we prove Theorem~\ref{theorem:trimmed}. We first note that using standard limit theorems in probability theory, we can show that $$\frac{S_k - M_k}{k \log k} \xrightarrow{d} 1.$$ The random variables $X_n$ are in the domain of attraction of a \emph{stable law} of order $1$. Using~\cite[Theorem 11.2.3]{AL}, we have $$\frac{S_k - k\log k}{k} \xrightarrow{d} C,$$ where $C$ is a Cauchy random variable. Thus, $$\frac{S_k}{k \log k}  = \left(\frac{1}{\log k}\frac{(S_k - k \log k)}{k} + 1 \right) \xrightarrow{d} 1.$$

Next, note that for $c >0$, $$\lim_{k \rightarrow \infty} P\left(\frac{M_k}{k} < c\right) = \lim_{k \rightarrow \infty} \left( 1 - \frac{1}{ck}\right)^{k} = e^{-1/c}.$$ Therefore $$\frac{M_k}{k\log k} \xrightarrow{d} 0.$$

Putting these together, we obtain $$\frac{S_k - M_k}{k \log k} \xrightarrow{d} 1,$$ and therefore we also have this convergence in probability.

To replace convergence in probability with almost sure convergence, we use a result of Mori~\cite[Theorem 1]{Mori}, with, in his notation, $$r=1, A(x) = x \log x.$$ Mori studies the behavior of the sample sums with the first $r$ largest terms removed, normalized by $A(n)$, where $A$ is an absolutely continuous increasing function such that there is an $0< \alpha< 2$, with $$A(x)x^{-\frac{1}{\alpha}} \mbox{ increasing and } \sup \frac{A(2x)}{A(x)} < \infty.$$ Keeping our notation consistent with ~\cite[Section 1]{Mori}, we write $A(x) = x \log x$, and note that this is a non-decreasing absolutely continuous function, with inverse function $B(x) = \frac{x}{W(x)}$, where $W(x)$ is the Lambert $W$ function, that is, the inverse of the function $we^w$. Putting $$\mathcal F(x) = P(X_1 > x) = \frac{1}{\lfloor x \rfloor + 1}$$ and $$J_2 = \int_{0}^{\infty} \mathcal F^2(x) dB^2(x) = \sum_{n=2}^{N} \frac{1}{n^2} \left( \frac{n^2}{W(n)^2} - \frac{(n-1)^2}{W(n-1)^2}\right),$$ Mori's result shows that if $J_2 < \infty$, that almost surely $$\frac{S_k-M_k}{A(k)} - c_k  \rightarrow 0,$$ where $$c_k = \frac{k}{A(k)} E(X_1| X_1 < A(k)).$$ In our setting, $$c_k \sim \frac{k}{k \log k} \log(k \log k) \rightarrow 1 \mbox{ as } k \rightarrow \infty.$$ So if we can show $J_2 < \infty$, we have completed our proof of Theorem~\ref{theorem:trimmed}. From~\cite{lambertw}, we see that for $x >>1$, $$W(x) = \log x - \log \log x  + o(1).$$ Note that $$dB^2(x) = 2 B(x) B'(x) dx = \frac{x}{W(x) (W(x) + 1)} dx.$$ Thus we have $$J_2 \le \int_1^{\infty} \frac{2}{x^2} B(x) B'(x) dx \sim 2 \int_{1}^{\infty} \frac{1}{x W(x) (W(x) + 1)}  =  2 \left(\frac{1}{W(1)}\right).$$ In particular,  $J_2 < \infty$. In Figure~\ref{fig:lambert}, we show the partial sums for $J_2$.

\begin{figure}[h!]\includegraphics[width=0.9\textwidth]{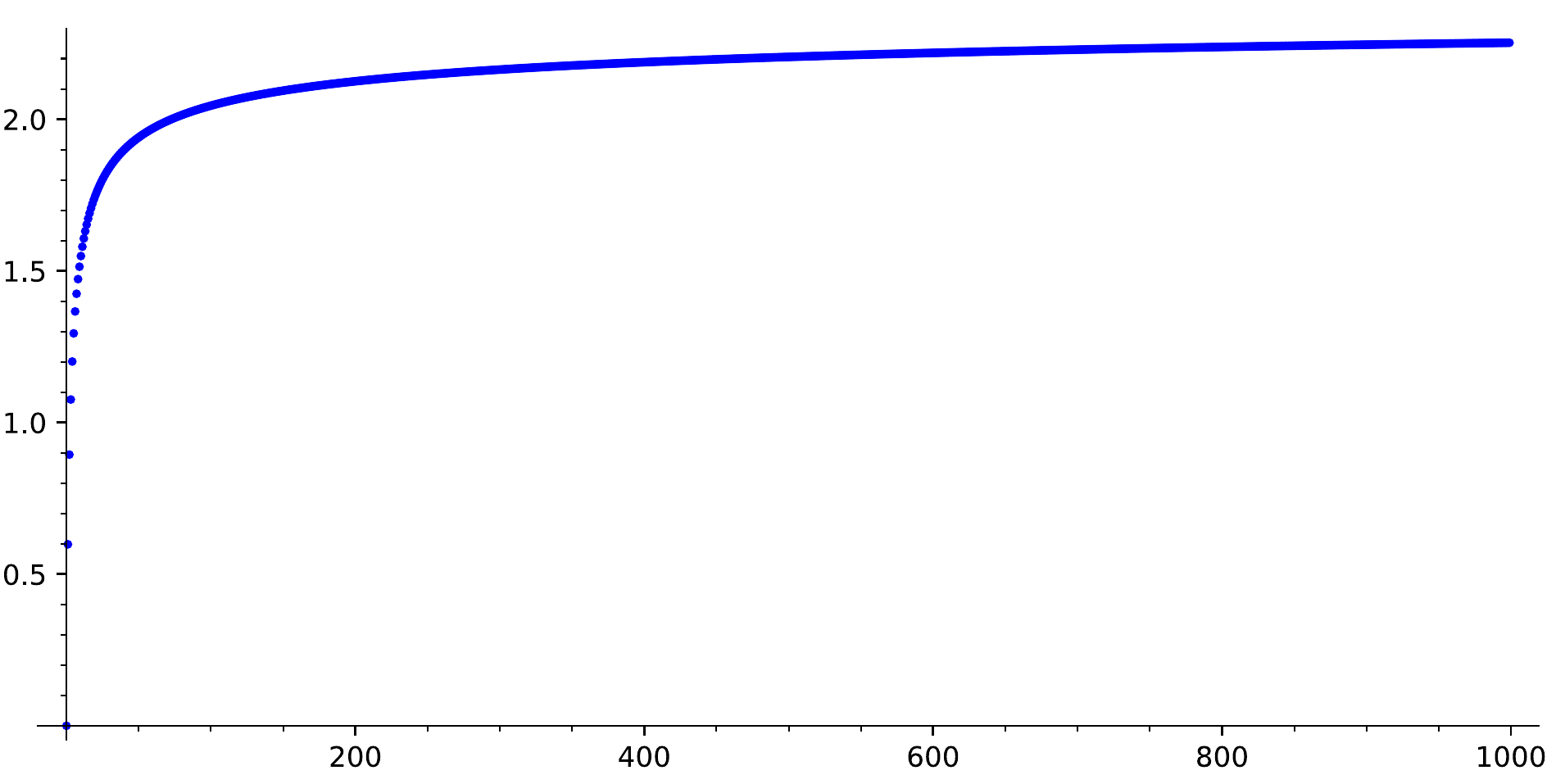}\caption{Partial sums of $\sum_{n=2}^{N} \frac{1}{n^2} \left( \frac{n^2}{W(n)^2} - \frac{(n-1)^2}{W(n-1)^2}\right)$, $N=3, \ldots, 1000$. \medskip}\label{fig:lambert}

\end{figure}

\section{Continued Fractions}\label{sec:Gauss} \noindent Similar questions can be asked for the \emph{Gauss map} $G: (0, 1) \rightarrow (0, 1)$, $$G(x) = \left\{ \frac 1 x \right\},$$ and the associated continued fraction expansion of a number $x \in [0, 1)$. Recall that for $x \in (0, 1)$, we can set, for $n \geq 1$, $$a_n = a_n(x) = \left \lfloor \frac{1}{G^{n-1}(x)} \right \rfloor,$$ and we can write $$x = \frac{1}{a_1 + \frac{1}{a_2 + \frac{1}{a_3 + \ldots}}}.$$ The probability measure $\mu_G$ given by $$d\mu_G(x) = \frac{1}{\log 2} \frac{dx}{1+x}$$ is an absolutely continuous ergodic invariant measure for $G$, and as noted after Theorem~\ref{theorem:trimmed}, Diamond and Vaaler~\cite{DV} proved a similar result for the digits $a_n$ for an $x$ chosen at random according to $\mu_G$. This is more difficult than our setting since although the digits form a stationary sequence, they are no longer IID. Regarding the uniqueness of sample maxima, if we define $$\rho_k = \mu_G\left(x \in (0, 1): \mbox{ there exists } ! 1 \le i \le k \mbox{ such that } a_i (x)= \max_{1 \le j \le k} a_j(x) \right),$$ we conjecture that $$\lim_{k \rightarrow \infty} \rho_k = 1.$$ We believe this conjecture since although the sequence $\{a_i\}$ is not IID, it is rapidly mixing, and the distribution of $a_i$ is heavy-tailed. The Gauss map is associated in a natural way to geodesic flow on the modular surface, and we believe similar results should also hold for other continued fraction type expansions associated to geodesic flow on other hyperbolic surfaces, for example, the Rosen continued fractions.

\bigskip

 \noindent {\bf Acknowledgements.} We thank Bruce Berndt, Alex Kontorovich, and Jeffrey Lagarias for useful discussions. Lagarias pointed out to us some interesting asymptotics involving binomial coefficients and zeta functions in~\cite[Theorem 2]{BL}. We thank Avanti Athreya, the anonyomous referee, and the Integers Editorial Office for suggestions that greatly improved the exposition of the paper. The first author thanks the National Science Foundation for its support via grant NSF DMS 2003528. The first author is a faculty member at the University of Washington, which is on the lands of the Coast Salish people, and touches the shared waters of all tribes and bands within the Suquamish, Tulalip and Muckleshoot nations.

\end{document}